\documentclass[11pt,a4paper,reqno]{amsart}
\usepackage{amssymb,amscd,amsmath, young, stmaryrd}
\usepackage{mathrsfs, mathdots,tikz,mathabx,enumerate,paralist,scalerel,relsize} 
\usepackage[mathcal]{eucal} 
\usepackage{float}
\usepackage{color}
\usepackage{url}
\usepackage{hyperref}
  \theoremstyle{plain}
  \newtheorem{theorem}{Theorem}[section]
  \newtheorem{lem}[theorem]{Lemma}
  \newtheorem{pro}[theorem]{Proposition}
  \newtheorem{cor}[theorem]{Corollary}

  \newtheorem{thmABC}{Theorem}

  \newtheorem*{con*}{Conjecture}
  
  \newtheorem{con}[thmABC]{Conjecture}

  \theoremstyle{remark}

  \newtheorem{rmk}[theorem]{Remark}

\theoremstyle{definition}
  \newtheorem{dfn}[theorem]{Definition}
  \newtheorem*{acknowledgements}{Acknowledgements}

  \numberwithin{equation}{section}
  \numberwithin{table}{section}

  \DeclareMathOperator*{\Ast}{\mathlarger{\scalerel*{{\star}}{\sum}}}
   \DeclareMathOperator{\des}{des}
  \DeclareMathOperator{\maj}{maj} \DeclareMathOperator{\Des}{Des}
   
  \DeclareMathOperator{\inv}{inv}

  \DeclareMathOperator{\imv}{imv} \DeclareMathOperator{\std}{std}
  \DeclareMathOperator{\kk}{k} \DeclareMathOperator{\den}{iden}
  \DeclareMathOperator{\denh}{den} \DeclareMathOperator{\exc}{exc}
  \DeclareMathOperator{\Exc}{Exc} \DeclareMathOperator{\N}{{\bf nexc}}
  \DeclareMathOperator{\E}{{\bf exc}} \DeclareMathOperator{\id}{id}
   \DeclareMathOperator{\iexc}{iexc}
  \DeclareMathOperator{\excabs}{exc^{abs}}
  \DeclareMathOperator{\negg}{neg} \DeclareMathOperator{\nden}{nden}
  \DeclareMathOperator{\nmaj}{nmaj} \DeclareMathOperator{\ndes}{ndes}
  \DeclareMathOperator{\fmaj}{fmaj} \DeclareMathOperator{\fdes}{fdes}
  \DeclareMathOperator{\dmaj}{dmaj} \DeclareMathOperator{\ddes}{ddes}
  \DeclareMathOperator{\DNeg}{DNeg} \DeclareMathOperator{\dneg}{dneg}
  \DeclareMathOperator{\dden}{dden} \DeclareMathOperator{\dexc}{dexc}
  \DeclareMathOperator{\nsp}{nsp}

\def\0{\hspace{-.12em}0}
\definecolor{dgreen}{HTML}{03297A}

\def\Nd{N}

\author[A.~Carnevale]{Angela Carnevale}
\address{School of Mathematics, Statistics and Applied Mathematics, National University of Ireland, Galway, Ireland}
	\email{angela.carnevale@nuigalway.ie}
	
\author[E.~Tielker]{Elena Tielker}
\address{Fakult\"at f\"ur Mathematik, Universit\"at Bielefeld, D-33501 Bielefeld, Germany
}
\email{etielker@math.uni-bielefeld.de}

 \title{On Denert's statistic} \date{}

\begin{document}
\maketitle
\thispagestyle{empty}

\begin{abstract}We show that the numerators  of genus zeta function associated with local hereditary orders studied by Denert can be described in terms of the joint
  distribution of Euler-Mahonian statistics on multiset permutations
  defined by Han. We use this result to deduce a reciprocity property
  for genus zeta functions of local hereditary orders whose associated
  composition is a rectangle. We also record a remarkable identity
  satisfied by genus zeta functions of local hereditary orders in
  terms of Hadamard products of genus zeta functions of maximal
  orders. Finally, we define Mahonian companions of excedance
  statistics on groups of signed and even-signed permutations.
  \end{abstract}

  \section{Introduction}
Recently, permutation statistics have found  applications  to various zeta functions in algebra;
see, for instance,~\cite{Berman/20,BC/17,CSV/18,SV/13,StasinskiVoll/14}.  An early instance of such applications arose from the
enumeration of ideals in hereditary orders encoded in so-called genus
zeta functions. It is known that local hereditary orders are parameterised by local invariants, which are integer compositions. In order to give an explicit expression for the
numerators of   genus zeta functions of such orders, Denert \cite{Denert/90}
defined a pair of statistics over  permutations.

Remarkably, for ``minimal'' (i.e.\ associated with the all-one composition) hereditary orders, the numerators of the associated genus zeta functions are, for a suitable choice of variables, Euler-Mahonian
polynomials over symmetric groups. This was first conjectured by
Denert in \cite{Denert/90} and then proved by Foata and Zeilberger in
\cite{FZ/90}.

Inspired by Denert's paper, Han~\cite{Han/91} gave a definition of a
\emph{Denert statistic} for multiset permutations, which together
with the classical excedance statistic is Euler-Mahonian. While Han's
result provides a Mahonian companion for the excedance statistic
already considered by MacMahon \cite{MacMahon/04} on multiset
permutations, it does not, to the best of our knowledge, provide a  combinatorial
interpretation of the numerators of Denert's genus zeta functions;  cf.~\cite[p.\ 25]{Han/91}.

This paper is devoted to a further study of Denert's statistic.
In the first part, we close the circle by showing that Denert's pair of
statistics (as originally defined) is indeed equidistributed with the
Euler-Mahonian statistics considered by Han on multiset
permutations; cf.~Theorem~\ref{Euler-Mahonian}. This gives an explicit description of the numerators of
the genus zeta functions of local hereditary orders with arbitrary
local invariants.

By results going back to MacMahon, our equidistribution result also implies a remarkable
identity involving Hadamard products of genus zeta functions of local
hereditary orders. Similar identities, also involving Eulerian
or Euler-Mahonian polynomials, have appeared in recent work on
so-called ask zeta functions \cite{Rossmann/18,cico} and zeta
functions associated with  quiver representations
\cite{LeeVoll/20}.

Generalisations of Euler-Mahonian identities to signed and even-signed permutations have been extensively studied (see, e.g.,~\cite{AdBrRo/01,Biagioli/03,Carnevale/17}). The remainder of this paper is devoted to generalisations of Denert's statistic which provide Mahonian companions to suitable excedance statistics on Coxeter groups of type $B$ and $D$.

 \smallskip
 
The paper is organised as follows. In Section~\ref{sec:preliminaries}
we collect some notation and preliminaries on permutation statistics on
multiset permutations, while in Section~\ref{sec:denstat} we recall
Denert's definitions of the statistics appearing in the numerators of
the genus zeta functions studied in \cite{Denert/90}.
Section~\ref{sec:main} is devoted to proving that these numerators are
indeed Euler-Mahonian polynomials.  In Section~\ref{sec:bd}, we define
analogues of Denert's statistics in types $B$ and  $D$. Together
with suitable excedance statistics, these are equidistributed with
Euler-Mahonian statistics on groups of signed and even-signed permutations, respectively. We conclude the paper with a few remarks in Section~\ref{sec:final}, including the aforementioned identity involving Hadamard products
satisfied by Denert's genus zeta functions.

\section{Notation and  preliminaries}\label{sec:preliminaries}

We set $[n]=\{ 1,\dots ,n \}$ and denote by $\{i_1,\dots,i_m\}_<$ a
set of increasing integers $i_1 < \dotsb < i_m$. We let $|S|$ denote the cardinality of a set $S$. For the remainder of this paper,  $\eta=(\eta_1,\dots ,\eta_r)$ is a fixed
composition of $n\in \mathbb{N}$ with $r$ parts. Given $\eta$, we let
$S_\eta$ denote the set of all permutations of the multiset
$$\{\underbrace{1,\dots ,1}_{\eta_1}, \dots ,\underbrace{r,
  \dots ,r}_{\eta_r}\}$$ comprising $\eta_1$ copies of $1$, $\eta_2$
copies of $2$, and so on.  In other words, a \emph{multiset
  permutation} in $S_\eta$ is a rearrangement of the ``trivial'' word
$\id^\eta={1}^{\eta_1}\cdots{r}^{\eta_r}\in S_{\eta}$.  Note that when
$\eta=(1,1,\dots ,1)$, $S_\eta$ is  the symmetric group $S_r$.  We will be
interested in several statistics on multiset permutations. 
We denote the \emph{descent set} of   $w = w_1\dotsb w_n\in S_{\eta}$ by
\begin{align*}
\Des(w)=\{ i\in [n-1] : w_i > w_{i+1} \}.
\end{align*}
The \emph{descent} and \emph{major index} statistics are
\begin{align*}
\des(w) = |\Des(w)| \quad \text{and} \quad \maj(w) = \sum_{i\in \Des(w)} i.
\end{align*}
Further, we define the descent set of a composition
$\Des(\eta):=\{\eta_1,\eta_1+\eta_2,\dots,\sum_{i=1}^{r-1}\eta_i\}$.

In the following, we recall a few definitions in order to define the pair of
statistics $(\denh, \exc)$, see also \cite{Han/91}. When $\eta$ is fixed, we will simply denote with $\id$ the trivial word $\id^\eta$ of the corresponding
set of multiset permutations.

A position $i\in [n]$ is an \emph{excedance} of $w\in S_{\eta}$ if the
$i$-th letter of $w$ is strictly greater than the $i$-th letter of the
trivial word $\id$.  We denote with $\Exc(w)$ the set of all excedances
of $w$ and with $\exc(w)$ its cardinality, viz.
\begin{equation}\label{def:exc}
\Exc(w) = \{ i\in [n] : w_i>\id_i \} \quad \mbox{and} \quad \exc(w)=|\Exc(w)|.
\end{equation}

\begin{dfn}
  Let $w\in S_{\eta}$.  
  The \emph{exceeding subword} of $w$ is 
  \begin{align*}
  \E(w):=w_{i_1}\cdots w_{i_k} \text{ for }\Exc(w)=\{i_1,\dots,i_k\}_<.
  \end{align*}
  The \emph{non-exceeding} subword of $w$ is
  \begin{align*}
  \N(w):=w_{j_1}\cdots w_{j_{n-k}} \text{ for }\{j_1,\dots,j_{n-k}\}_<:=[n]\setminus \Exc(w).
\end{align*}
\end{dfn}
For example, for $\eta= (3,2,2,3)$ and $w=4232314141$, the exceeding
subword is $\E(w)= 42334$ and the non-exceeding  subword is $\N(w)=21141$.

As usual,   we let $\inv(w)$ denote the
\emph{inversion number} of  a multiset permutation $w\in S_\eta$ $$\inv(w)=|\{(i,j):1\leq i<j\leq n,\,
w_i>w_j\}|$$ and $\imv(w)$ denote the \emph{weak inversion number}  of
$w$ $$\imv(w)=|\{(i,j):1\leq i<j\leq n,\, w_i\geq w_j\}|.$$
Generalising work of Foata and Zeilberger on permutations \cite{FZ/90}, Han gave the
following definition of a Denert statistic on multiset permutations.
\begin{dfn}[{\cite[D\'efinition 1.3]{Han/91}}]\label{def:denh} 
  Let $w\in S_{\eta}$. Denert's statistic on multiset permutations is given by
  $$\denh(w):=\sum_{i\in \Exc(w)} i +\imv(\E(w))+\inv(\N(w)).$$
  \end{dfn}
 For instance, $\denh(4232314141)=18+5+4=27$. Han proved that this
 statistic, together with the excedance number defined in
 \eqref{def:exc}, is equidistributed with the pair of statistics
 $(\maj,\des)$ on multiset permutations.
\begin{theorem}[{\cite[Th\'eor\`eme 1.4]{Han/91}}]\label{Theorem_Han}
The pair of statistics on $S_{\eta}$ $(\denh, \exc)$ is Euler-Mahonian, i.e.\
\begin{align*}
\sum_{w\in S_{\eta}} x^{\denh(w)} y^{\exc(w)} = \sum_{w\in S_{\eta}} x^{\maj(w)} y^{\des(w)}.
\end{align*}
\end{theorem}

\section{Denert's statistic}\label{sec:denstat}
Our first main result shows that the polynomials expressing the
numerators of the genus zeta functions of hereditary orders with local
invariants $\eta$ and $r$ coincide with the polynomials giving the
joint distribution of $(\denh,\exc)$ over $S_\eta$ in Theorem~\ref{Theorem_Han}. These numerators,
as defined by Denert in \cite[Theorem 11]{Denert/90}, involve
statistics on so-called $\eta$-admissible permutations, $\den$ and
$\iexc$, which we now define, closely following \cite{Denert/90}.

Let $\sigma\in S_n$. Following Denert, we visualise  $\sigma$ as the
matrix whose  $(i,j)$-th entry is defined as
\begin{align*}
(i,j)= \begin{cases}
1\quad &\text{ if }j=\sigma(i),\\
0\quad &\text{ otherwise.}
\end{cases}
\end{align*}
Note that this is the transpose of the usual permutation matrix
associated with $\sigma$. Nevertheless, to ease the translation
between Denert's and our notation, we will refer to it
as the matrix associated with $\sigma$. 
Since we are interested in statistics counting certain zero entries,
we think of this matrix as an  $n\times n$ grid, and we
refer to matrix entries as cells in this grid.

\begin{figure}[b]
    \begin{tikzpicture}[scale=.5]
      \draw[step=1.0,gray,very thin,xshift=-0.5cm,yshift=-0.5cm] (1,1)
      grid (11,11);
\begin{scope}
   \foreach \x in {(1,1), (1,2), (1,3), (1,4), (1,5), (1,6), (1,7),
    (2,1), (2,2), (2,3), (2,4), (2,5), (2,6), (2,7), (3,1), (3,2),
    (3,3), (3,4), (3,5), (3,6), (3,7), (4,1), (4,2), (4,3), (4,4),
    (4,5), (5,1), (5,2), (5,3), (5,4), (5,5), (6,1), (6,2), (6,3),
    (7,1), (7,2), (7,3)} {\draw[draw=gray,very
    thin,fill=black!20!,xshift=-0.5cm,yshift=-0.5cm,] \x rectangle
    ++(1,1);} \draw[thick,xshift=-0.5cm,yshift=-0.5cm]
  (1,8)--(4,8)--(4,6)--(6,6)--(6,4)--(8,4)--(8,1);
\end{scope}
\end{tikzpicture}%
\caption{For $n=10$ and $\eta=(3,2,2,3)$, the set  $[\succ]$ is
  coloured in grey, while the set  $[\preceq]$ is left blank.} \label{fig:path}
\end{figure}
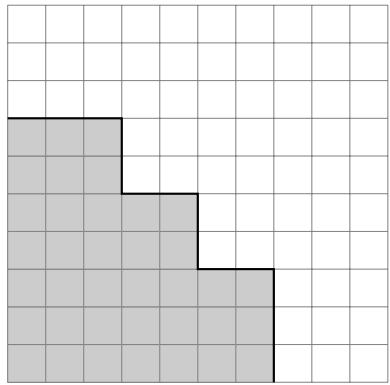

\begin{dfn}
  The \emph{projection} or \emph{block-map} with respect to the composition
  $\eta$ is the map $\pi_{\eta}:[n]\to [r]$ such that
\begin{align*}
 \sum_{k=1}^{\pi_{\eta}(i)-1}\eta_i< i \leq \sum_{k=1}^ {\pi_{\eta}(i)}\eta_i.
\end{align*}
 That is, $\pi_{\eta}(i)=1$ for $1\le i \le \eta_1$,
$\pi_{\eta}(i)=2$ for   $\eta_1 +1\le i \le \eta_1+\eta_2$  and so on.

By slight abuse of notation, we also denote by
$\pi_{\eta}\colon S_n \to S_{\eta}$ the projection from permutations
to multiset permutations
\begin{align*}
\pi_{\eta}(\sigma):= \pi_{\eta}(\sigma(1)) \dotsb \pi_{\eta}(\sigma(n)).
\end{align*}
\end{dfn} 
 For instance, $\pi_{(3,2,2,3)}(681\02435179) = 3441212134$.

 The block-map partitions a permutation matrix into $r^2$ blocks of
 size $\eta_i \times \eta_j$, $1\leq i,j \leq r$.
   For $k \in \mathbb{N}$ we define the \emph{$k$-th
   block-row} (resp.\ \emph{$k$-th block-column}) to be the set of
 pairs $(i,j)\in [n]^2$ such that $\pi_{\eta}(i)=k$ (resp.\
 $\pi_{\eta}(j)=k$).  Let further 
 \begin{align*}
   [\preceq ]&=\{(i,j):\pi_\eta(i)\leq \pi_\eta(j)\},\\
   [\prec ]&=\{(i,j):\pi_\eta(i)< \pi_\eta(j)\},\\
  [\succ ]&=\{(i,j):\pi_\eta(i)> \pi_\eta(j)\}.
\end{align*} 
We illustrate the sets $[\preceq]$ and $[\succ]$ in
Figure~\ref{fig:path}, see also \cite[Section
1]{Denert/90}.
Following Denert, we say that a permutation $\sigma\in S_n$ is
\emph{descending on $I\subseteq [n]^2$} if for all $(i,\sigma(i))$,
$(j,\sigma(j))\in I$, $i < j$ if and only if $\sigma(i) < \sigma(j)$.
For instance, $\sigma=681\02435179\in S_{(3,2,2,3)}$ is descending on
every block-row, but not on the first and last block-column, which can
be easily seen in Figure \ref{fig:admissible}.

The polynomials we are interested
in are generating polynomials on permutations which Denert calls
\emph{$\eta$-admissible permutations}. These are permutations whose
descent sets are contained in the descent set of the composition
$\eta$.

\begin{dfn} 
  A permutation $\sigma\in S_n$ is \emph{$\eta$-admissible} if it is
  descending on every block-row.  We will denote
  $S^{\eta}=\{\sigma \in S_n : \Des(\sigma)\subset\Des(\eta)\}$ the
  set of all $\eta$-admissible permutation in $S_n$.
\end{dfn}
For instance, $\sigma=681\02435179$ 
is
$(3,2,2,3)$-admissible, while  $\tau=681\04235179$ is not (see also Figure~\ref{fig:admissible}).
Note that the set of $\eta$-admissible permutations is a
  parabolic quotient of $S_n$; see, e.g., \cite[Section~2.4]{BjBr/05}.

\begin{figure}[bt]
 \begin{minipage}[b]{.45\linewidth} 
        \begin{center}
 \begin{tikzpicture}[scale=.5]
   \draw[step=1.0,gray,very thin,xshift=-0.5cm,yshift=-0.5cm] (1,1) grid (11,11);
 \begin{scope}
   \draw (6,10)node{$1$}; \draw (8,9)node{$1$}; \draw (10,8)node{$1$};
   \draw (2,7)node{$1$}; \draw (4,6)node{$1$}; \draw (3,5)node{$1$};
   \draw (5,4)node{$1$}; \draw (1,3)node{$1$}; \draw (7,2)node{$1$};
   \draw (9,1)node{$1$}; \draw[thick,xshift=-0.5cm,yshift=-0.5cm]
   (1,8)--(4,8)--(4,6)--(6,6)--(6,4)--(8,4)--(8,1);
   \draw[dashed,thick,xshift=-0.5cm,yshift=-0.5cm] (4,11)--(4,8)
   (4,8)-- (11,8) (6,11)--(6,6) (6,6)--(11,6) (8,11)--(8,4)
   (8,4)--(11,4) (1,6)--(4,6) (4,6)--(4,1) (1,4)--(6,4) (6,4)--(6,1);
 \end{scope}
 \end{tikzpicture}
 	\end{center}
 \end{minipage}
 \begin{minipage}[b]{.45\linewidth} 
        \begin{center}
 \begin{tikzpicture}[scale=.5]
   \draw[step=1.0,gray,very thin,xshift=-0.5cm,yshift=-0.5cm] (1,1) grid (11,11);
 \begin{scope}
 \draw (6,10)node{$1$};
 \draw (8,9)node{$1$};
 \draw (10,8)node{$1$};
 \draw (4,7)node{$1$};
 \draw (2,6)node{$1$};
 \draw (3,5)node{$1$};
 \draw (5,4)node{$1$};
 \draw (1,3)node{$1$};
 \draw (7,2)node{$1$};
 \draw (9,1)node{$1$};
 \draw[thick,xshift=-0.5cm,yshift=-0.5cm] (1,8)--(4,8)--(4,6)--(6,6)--(6,4)--(8,4)--(8,1);
 \draw[dashed,thick,xshift=-0.5cm,yshift=-0.5cm] (4,11)--(4,8) (4,8)-- (11,8) (6,11)--(6,6) (6,6)--(11,6) (8,11)--(8,4) (8,4)--(11,4) (1,6)--(4,6) (4,6)--(4,1) (1,4)--(6,4) (6,4)--(6,1);
 \end{scope}
 \end{tikzpicture}
 	\end{center}
 \end{minipage}
 \caption{Let $n=10$ and $\eta=(3,2,2,3)$. The left matrix corresponds
   to $\sigma= 681\02435179$ and the right matrix to
   $\tau=681\04235179$}
   \label{fig:admissible}
 \end{figure}
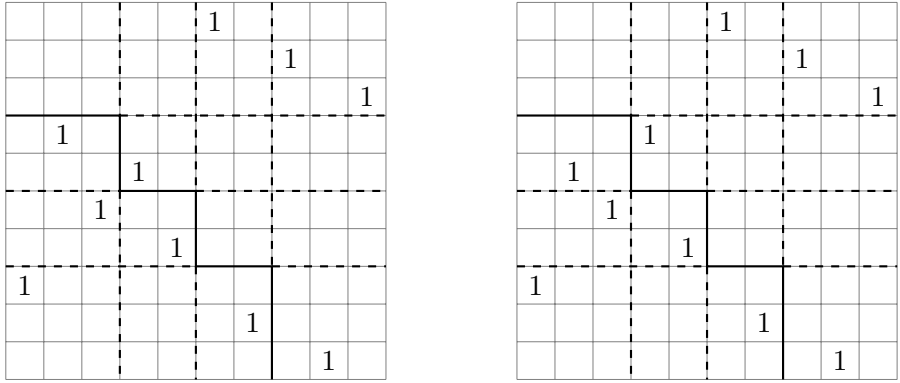

  It is well known that parabolic quotients and thus $\eta$-admissible
 permutations are in bijection with the set of multiset permutations
 $S_{\eta}$ via the map $ \sigma \mapsto \pi_\eta (\sigma^{-1}).$
 Indeed, the projection $\pi_\eta$ is injective on the set of
 permutations whose inverses have  descent sets contained in
 $\Des(\eta)$. The inverse of this map is defined in terms of the
 \emph{standardisation} $\std=\std_\eta:S_{\eta} \to S_n$. Informally,
 the standardisation of $w\in S_{\eta}$ is a permutation $\std(w)$
 which we obtain from $w$ by substituting the $\eta_1$ $1$s from left to right
 with $1, \dots, \eta_1$, the $\eta_2$ $2$s from left to right with
 $\eta_1+1, \dots, \eta_1+\eta_2$ and so on; see
 also, e.g.,~\cite[Section 2]{Carnevale/17}. We then obtain an $\eta$-admissible permutation by taking the inverse of $\std(w)$. That is, \begin{align}
                                             S^\eta &\overset{1-1}{\longleftrightarrow} S_\eta \nonumber\\
                                             \sigma &\mapsto \pi_\eta (\sigma^{-1})\label{eq:bijection}\\
                                             (\std_\eta (w))^{-1}
                                                    &\mapsfrom
                                                      w. \nonumber
\end{align}

For instance, for $\eta=(3,2,2,3)$ and
$\sigma= 681\02435179$, we have
$\sigma^{-1}=846571921\03$ and thus $\pi_{\eta}(\sigma^{-1})=4232314141$.
On the other hand, $\std(4232314141)= 846571921\03=\sigma^{-1}$, and
therefore 
$(\std(4232314141))^{-1}= \sigma$, as claimed.

This bijection is a key ingredient in the proof of
Theorem \ref{Euler-Mahonian}. We are now ready to introduce the first
of the two statistics needed to show our main result.
\begin{dfn}
For $\sigma \in S_n$ and $\eta$ a composition of $n$ we define  
\begin{align*}
 I_\sigma=\{(i,\sigma(i))\in [\succ]\}=\{j: \pi_\eta(\sigma^{-1}(j)) > \pi_\eta(j)\}.
\end{align*}\end{dfn}
Note that $I_\sigma$ coincides with the set of  excedances of $\pi_\eta(\sigma^{-1}))$,
that is
 $I_\sigma=\Exc(\pi_\eta(\sigma^{-1}))$. Therefore, we denote its
 cardinality with
\begin{align*}
\iexc(\sigma) := |I_\sigma|.
\end{align*}
\begin{rmk}
  The statistic $\iexc$ appears as $\kk$ in \cite{Denert/90}.
  \end{rmk}

 \begin{figure}[t]
\begin{tikzpicture}[scale=.5]
  \draw[step=1.0,gray,very thin,xshift=-0.5cm,yshift=-0.5cm] (1,1) grid (11,11);
\begin{scope}
  \draw (6,10)node{$1$}; \draw (8,9)node{$1$}; \draw (10,8)node{$1$};
  \draw (2,7)node{$1$}; \draw (4,6)node{$1$}; \draw (3,5)node{$1$};
  \draw (5,4)node{$1$}; \draw (1,3)node{$1$}; \draw (7,2)node{$1$};
  \draw (9,1)node{$1$}; \foreach \x in
  {(6,7),(6,6),(6,5),(6,4),(8,7),(8,6),(8,5),(8,4),(8,3),(8,2),(10,7),(10,6),
    (10,5),(10,4),(10,3),(10,2),(10,1)}
  {\draw[draw=blue,fill=blue!40!,xshift=-0.5cm,yshift=-0.5cm,] \x
    rectangle ++(1,1);} \foreach \x in {(3,7),(5,5),(7,3)}
  {\draw[draw=red,fill=red!40!,xshift=-0.5cm,yshift=-0.5cm,] \x
    rectangle ++(1,1);} \draw[thick,xshift=-0.5cm,yshift=-0.5cm]
  (1,8)--(4,8)--(4,6)--(6,6)--(6,4)--(8,4)--(8,1);
  \draw[dashed,thick,xshift=-0.5cm,yshift=-0.5cm] (4,11)--(4,8)
  (4,8)-- (11,8) (6,11)--(6,6) (6,6)--(11,6) (8,11)--(8,4)
  (8,4)--(11,4) (1,6)--(4,6) (4,6)--(4,1) (1,4)--(6,4) (6,4)--(6,1);
\end{scope}
\end{tikzpicture}
\caption{Let $n=10$, $\eta=(3,2,2,3)$ and $\sigma=
  681\02435179$. Elements of  $\Nd^+_\sigma$ are marked in
  blue and elements of $\Nd^-_\sigma$ are marked in red.}
\label{fig:Ns}
\end{figure}
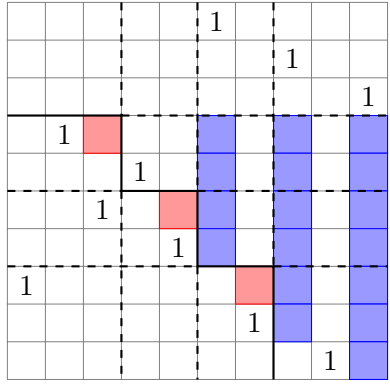

Further, we give here  the  definitions of the sets $\Nd^+_\sigma$ and $\Nd^-_\sigma$,
$$\Nd^+_\sigma=[\preceq ]\cap\{(i,j): \sigma(i)<j \mbox{ and }\sigma^{-1}(j)<i\},$$
and 
$$\Nd^-_\sigma=[\succ ]\cap\{(i,j): \sigma(i)<j \mbox{ and }\sigma^{-1}(j)>i\}.$$
Note that Denert uses the same notation for the cardinalities of these
sets; cf.~\cite[Section~2]{Denert/90}.
Figure~\ref{fig:Ns} illustrates $\Nd_{\sigma}^+$ and $\Nd_{\sigma}^-$
for a permutation in $S_{(3,2,2,3)}$, where we marked elements of
$\Nd_{\sigma}^+$ and $\Nd_{\sigma}^-$ as coloured cells in the
permutation matrix of $\sigma$.\

The statistic introduced in the next definition implicitly appeared in
the numerators of Denert's genus zeta functions. For this reason, we
refer to it as Denert's statistic (see also
Proposition~\ref{genus_zeta_function}).
\begin{dfn} 
For $\sigma\in S^\eta$,  Denert's statistic is defined as 
$$\den(\sigma) = \sum_{j\in I_\sigma} j + |\Nd^+_\sigma| - |\Nd^-_\sigma| - \iexc({\sigma}).$$
\end{dfn}
 For instance, for $\sigma=681\02435179$, $\den(\sigma)=18+17-3-5=27$,
 see also Figure \ref{fig:Ns}.
 
 Note that thanks to the map \eqref{eq:bijection}
 $\sigma\mapsto \pi_{\eta}(\sigma ^{-1})$, we obtain a statistic on
 the set of multiset permutations. Our goal is to show that the
 statistic obtained in this way is indeed Han's statistic from
 Definition~\ref{def:denh}, which justifies our notation.

 As mentioned above, Denert's
 statistic appears in the numerators of  genus zeta
 functions of local hereditary orders.  In the next subsection, we recall the definition of such zeta functions and the main result of  \cite{Denert/90}.
\subsection{Genus zeta functions of local hereditary orders}
 For a composition $\eta$ of $n$, set
$$W_\eta(x,y):=\frac{\sum_{\sigma \in S^{\eta}} x^{\den (\sigma)} y^{\iexc({\sigma})}}{\prod_{0\leq j \leq n-1}(1-x^i y)} \in \mathbb Q(x,y).
 $$
 Then \cite[Theorem 11]{Denert/90} is a closed formula for the genus
 zeta function of a local hereditary order in terms of the rational
 functions $W_\eta$.

 We briefly recall here the relevant definitions, the aforementioned
 result and a sketch of its proof.

 Let $K$ be a local field and $R$ be its ring of integers. Let $A$ be
 a central simple algebra over $K$. Then $A$ is isomorphic to $M_n(D)$
 for a unique integer $n$ and division $K$-algebra $D$. Let $\Delta$
 be the unique maximal order in $D$ and let $\mathfrak p$ be the
 unique maximal two-sided ideal of $\Delta$. Write
 $q=|\Delta/\mathfrak p|$.

 Given an $R$-order $\Theta$ in $A$, the genus zeta function of
 $\Theta$ is the Dirichlet series $Z_\Theta(s)= \sum |\Theta:\mathcal
 L|^{-s},$ where the sum ranges over integral free ideals of $\Theta$; cf.\
 \cite[Definition 3.1]{Denert/90}.
It is known that hereditary orders in $A$ are parameterised by
so-called local invariants, which are 
compositions of $n$. Given any such composition $\eta$, an explicit
description of a
hereditary order $\Theta^\eta$ with local invariant parameterised by an integer composition $\eta$ can be
found in \cite[Theorem 7]{Denert/90}.
\begin{pro}\label{genus_zeta_function}
$\displaystyle
Z_{\Theta^\eta}(s) = W_\eta(q,q^{-ns}).$
\end{pro}
\begin{proof}[(Sketch of) Proof of Proposition
  \ref{genus_zeta_function}] Following  Denert's proof of \cite[Theorem
  11]{Denert/90}, we have
\begin{align*}
Z_{\Theta^\eta}(s) = \sum_{\sigma \in S^{\eta}}
  q^{|N_{\sigma}^+|-|N_{\sigma}^-|}
  \sum_{\substack{\lambda\in\mathbb{N}^n\\  \lambda_j>0 \text{ if
  }j\in I_{\sigma}}}  \prod_{1\leq j\leq n}(q^{j-1-ns})^{\lambda_j}.
\end{align*}
Setting $t:=q^{-ns}$, with an inclusion-exclusion argument we obtain 
\begin{align*}
  \sum_{\substack{\lambda\in\mathbb{N}^n\\ \lambda_j>0 \text{ if }j\in
  I_{\sigma}}}&  \prod_{1\leq j\leq n}(q^{j-1}t)^{\lambda_j} =
                \sum_{\lambda\in\mathbb{N}^n}  \prod_{1\leq j\leq n}(q^{j-1}t)^{\lambda_j} - \sum_{j\in I_{\sigma}} \sum_{\substack{\lambda\in\mathbb{N}^n\\ \lambda_j=0}} \prod_{1\leq j\leq n}(q^{j-1}t)^{\lambda_j}\\
              & + \sum_{\{j_1,j_2\}_<\subset I_{\sigma}}
                \sum_{\substack{\lambda\in\mathbb{N}^n\\  \lambda_{j_1}=\lambda_{j_2}=0}} \prod_{1\leq j\leq n}(q^{j-1}t)^{\lambda_j} - \dots (-1)^{|I_{\sigma}|} \sum_{\substack{\lambda\in\mathbb{N}^n\\ \lambda_j=0 \text{ if }j\in I_{\sigma}}} \prod_{1\leq j\leq n}(q^{j-1}t)^{\lambda_j}\\
              & = \left( \prod_{1\leq j \leq n} (1-q^{j-1}t) \right)^{-1} \left( 1 + \sum_{\emptyset \neq J\subseteq I_{\sigma}} (-1)^{|J|} \prod_{j\in J} (1-q^{j-1}t) \right)\\
              &= \left( \prod_{1\leq j \leq n} (1-q^{j-1}t) \right)^{-1} \prod_{j\in I_{\sigma}} q^{j-1}t.
\end{align*}
Therefore,
\begin{align*}
Z_{\Theta^\eta}(s) &= \frac{\sum_{\sigma \in S^{\eta}} q^{|N_{\sigma}^+|-|N_{\sigma}^-|} \prod_{j\in I_{\sigma}} q^{j-1-ns|I_{\sigma}|}}{\prod_{1\leq j \leq n} (1-q^{j-1-ns})}\\
&= \frac{\sum_{\sigma \in S^{\eta}} q^{|N_{\sigma}^+|-|N_{\sigma}^-|+
                                                                                                                                                                             \sum_{j\in I_{\sigma}} j-(1+ns)\iexc({\sigma})} }{\prod_{1\leq j \leq n} (1-q^{j-1-ns})}\\
  &=\frac{\sum_{\sigma \in S^{\eta}} q^{\den (\sigma)-ns\iexc({\sigma})}}{\prod_{0\leq i \leq n-1}(1-q^{i -ns})},
\end{align*}
as claimed.
\end{proof}

\section{Denert's genus zeta function and  Euler-Mahonian polynomials }\label{sec:main}
In this section we prove our theorem about the equidistribution
of $(\denh, \exc)$ over the set of multiset permutations $S_\eta$ and
that of $(\den, \iexc)$ over the set  of $\eta$-admissible
permutations $S^{\eta}$.

\begin{theorem} \label{Euler-Mahonian}
The pair of statistics $(\den, \iexc)$ is Euler-Mahonian, i.e.\
\begin{align*}
    \sum_{\sigma \in S^\eta} x^{\den(\sigma)} y^{\iexc(\sigma)} = \sum_{w \in S_\eta} x^{\denh(w)} y^{\exc(w)}.
\end{align*}
\end{theorem}

\noindent In preparation for the proof, we further partition the set
$\Nd^+_\sigma$ 
into
\begin{equation*}
  \Nd^+_\sigma [\preceq] = \{ (i,j) : \sigma(i)<j, \sigma^{-1}(j)<i, \pi_\eta(i) \leq \pi_\eta(j), \underbrace{\pi_\eta(i)\leq \pi_\eta(\sigma(i))}_{\text{i.e. }(i,\sigma(i))\in [\preceq]}\}
\end{equation*}
and
\begin{equation*}
    \Nd^+_\sigma [\succ] = \{ (i,j) : \sigma(i)<j, \sigma^{-1}(j)<i, \pi_\eta(i) \leq \pi_\eta(j), \underbrace{\pi_\eta(i) > \pi_\eta(\sigma(i))}_{\text{i.e. }(i,\sigma(i))\in [\succ]}\},
\end{equation*}
see Figure \ref{fig:NPlus} for an example. 
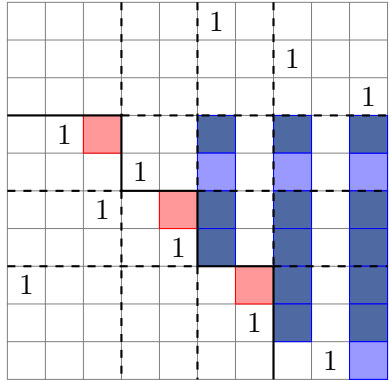
\begin{figure}[h]
\begin{center}
\begin{tikzpicture}[scale=.5]
  \draw[step=1.0,gray,very thin,xshift=-0.5cm,yshift=-0.5cm] (1,1) grid (11,11);
\begin{scope}
  \draw (6,10)node{$1$}; \draw (8,9)node{$1$}; \draw (10,8)node{$1$};
  \draw (2,7)node{$1$}; \draw (4,6)node{$1$}; \draw (3,5)node{$1$};
  \draw (5,4)node{$1$}; \draw (1,3)node{$1$}; \draw (7,2)node{$1$};
  \draw (9,1)node{$1$}; \foreach \x in {(6,7),(8,7)}
  {\draw[draw=blue,fill=dgreen!70!,xshift=-0.5cm,yshift=-0.5cm,] \x
    rectangle ++(1,1);} \foreach \x in {(6,6),(8,6),(10,6)}
  {\draw[draw=blue,fill=blue!40!,xshift=-0.5cm,yshift=-0.5cm,] \x
    rectangle ++(1,1);} \foreach \x in
  {(6,5),(6,4),(8,5),(8,4),(8,3),(8,2),(10,7),(10,5),(10,4),(10,3),(10,2)}
  {\draw[draw=blue,fill=dgreen!70!,xshift=-0.5cm,yshift=-0.5cm,] \x
    rectangle ++(1,1);} \foreach \x in {(10,1)}
  {\draw[draw=blue,fill=blue!40!,xshift=-0.5cm,yshift=-0.5cm,] \x
    rectangle ++(1,1);} \foreach \x in {(3,7),(5,5),(7,3)}
  {\draw[draw=red,fill=red!40!,xshift=-0.5cm,yshift=-0.5cm,] \x
    rectangle ++(1,1);} \draw[thick,xshift=-0.5cm,yshift=-0.5cm]
  (1,8)--(4,8)--(4,6)--(6,6)--(6,4)--(8,4)--(8,1);
  \draw[dashed,thick,xshift=-0.5cm,yshift=-0.5cm] (4,11)--(4,8)
  (4,8)-- (11,8) (6,11)--(6,6) (6,6)--(11,6) (8,11)--(8,4)
  (8,4)--(11,4) (1,6)--(4,6) (4,6)--(4,1) (1,4)--(6,4) (6,4)--(6,1);
\end{scope}
\end{tikzpicture}%
\caption{Let $n=10$ and $\eta=(3,2,2,3)$. For $\sigma= 681\02435179$ the set $\Nd^+_\sigma[\preceq]$ is marked in blue, the set $\Nd^+_\sigma[\succ]$ is marked in dark blue, while the set  $\Nd^-_\sigma$ is marked in red.}
\label{fig:NPlus}
\end{center}
\end{figure}

The following technical lemmata are key to show 
that  $ \den(\sigma) =\denh(\pi_\eta(\sigma^{-1}))$.
We show the latter identity as a result of finer identities, starting
with the
following.
\begin{lem} \label{Lem=<}
  Let $\eta$ be a composition of $n$ and   $\sigma \in S^{\eta}$. Then
\begin{align*}
     |\Nd^+_\sigma [\preceq]| = \inv(\N(\pi_\eta(\sigma^{-1}))).
\end{align*} 
\end{lem}
\begin{proof}
Since $\sigma\in S^\eta$, $\sigma$ is descending on every
block-row. Thus $\sigma^{-1}$ is descending on every block-column, that is
if $i<j$  with $\pi_\eta(\sigma^{-1}(i))= \pi_\eta(\sigma^{-1}(j))$, then $ \sigma^{-1}(i) < \sigma^{-1}(j)$;
see also Figure~\ref{fig:block-column}.
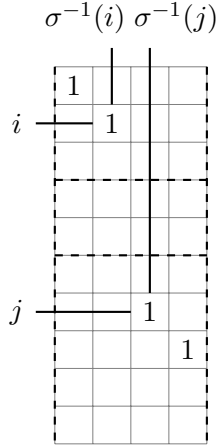
\begin{figure}[h]
  \begin{center}
\begin{tikzpicture}[scale=.5]
  \draw[step=1.0,gray,very thin,xshift=-0.5cm,yshift=-0.5cm] (3,1) grid (7,11);
\begin{scope}
  \draw (3,10)node{$1$}; \draw (4,9)node{$1$}; \draw (5,4)node{$1$};
  \draw (6,3)node{$1$}; \draw[thick,xshift=-0.5cm,yshift=-0.5cm]
  (2.5,9.5)--(4,9.5) (4.5,10)--(4.5,11.5) (2.5,4.5)--(5,4.5)
  (5.5,5)--(5.5,11.5); \draw[dashed,thick,xshift=-0.5cm,yshift=-0.5cm]
  (3,11)--(3,1) (3,8)--(7,8) (3,6)--(7,6) (7,11)--(7,1); \draw
  (1.5,9)node{$i$}; \draw (1.5,4)node{$j$}; \draw
  (3.3,11.8)node[]{$\sigma^{-1}(i)$}; \draw
  (5.7,11.8)node[]{$\sigma^{-1}(j)$};
\end{scope}
\end{tikzpicture}
	\end{center}
  \caption{A block-column of $\sigma^{-1}$.}
  \label{fig:block-column}
\end{figure}
But $\sigma^{-1}(i) > \sigma^{-1}(j)$ also implies 
$\pi_\eta(\sigma^{-1}(i))\geq \pi_\eta(\sigma^{-1}(j))$. Therefore, for $i<j$ we have
\begin{align}\label{eq:iff}
    \sigma^{-1}(i) > \sigma^{-1}(j) \Leftrightarrow \pi_\eta(\sigma^{-1}(i))> \pi_\eta(\sigma^{-1}(j)).
\end{align}
  By  definition, setting  $k=\sigma(i)$ and using Eq.~\eqref{eq:iff}, we get
\begin{align}
    |\Nd^+_\sigma [\preceq]|& =|\{ (k,j) : k<j, \sigma^{-1}(j)<\sigma^{-1}(k), \pi_\eta(\sigma^{-1}(k)) \leq \pi_\eta(j), \pi_\eta(\sigma^{-1}(k))\leq \pi_\eta(k)\}|\nonumber\\ \label{Eq.lem1.2}
                            &=    |\{ (k,j) : k<j, \pi_\eta(\sigma^{-1}(j))<\pi_\eta(\sigma^{-1}(k)) \leq \pi_\eta(j), \pi_\eta(\sigma^{-1}(k))\leq \pi_\eta(k)\}|.
\end{align}
Consider the non-exceeding subword of $\pi_\eta(\sigma^{-1})$
\begin{align*}
    \N(\pi_\eta(\sigma^{-1})) = \pi_\eta(\sigma(i_1))\cdots \pi_\eta(\sigma(i_m)),
\end{align*}
where $\pi_\eta(\sigma(i))\leq \pi_\eta(i)$ if and only if
$i\in \{i_1,\dots ,i_m\}_<=[n]\setminus\Exc(\pi_\eta(\sigma^{-1}))$. The lemma now follows by comparing  Eq.~\eqref{Eq.lem1.2} with
\begin{align*}
  \inv(\N(\pi_\eta(\sigma^{-1}))) = |\{ (i,j) :\ & i<j, \pi_\eta(\sigma^{-1}(j))<\pi_\eta(\sigma^{-1}(i)), \\ & \pi_\eta(\sigma^{-1}(j))<\pi_\eta(j), \pi_\eta(\sigma^{-1}(i))<\pi_\eta(i) \}|.\qedhere
\end{align*}
\end{proof}
We now give  a few more definitions that are needed for the next
lemma. For $l\in
\{2,\dots ,r\}$, following \cite[Section 1]{Denert/90} we set
\begin{align*}
    U_\sigma(l) &:= \{ (i,\sigma(i)) : l\leq \pi_\eta(i), \pi_\eta(\sigma(i))<l \}.
\end{align*}     
Let us further define
\begin{align*}
    U_\sigma^{-1}(l) &:= \{ (i,\sigma(i)) : \pi_\eta(i) < l, l\leq \pi_\eta(\sigma(i)) \}.
\end{align*}
The statistics $U_{\sigma}(l)$ and $  U_\sigma^{-1}(l)$ count,
respectively, the number of ones in certain north-east and south-west
quadrants of the grid,
see Figure \ref{fig:U} for an example.

For
$(j_0,\sigma(j_0))\in [\succ]$, we set 
\begin{equation}\label{eq:nsminus} N_\sigma^-(j_0) := [\succ ]\cap\{(j_0,i):\sigma(j_0)< i, j_0 <
\sigma^{-1}(i)\},\end{equation} and \begin{equation}\label{eq:nsplus}
    N_\sigma^+[\succ](j_0) :=  \{ (j_0,i) : \sigma(j_0)<i, \sigma^{-1}(i)<j_0, \pi_\eta(j_0) \leq \pi_\eta(i), \pi_\eta(j_0) > \pi_\eta(\sigma(j_0))\}.
\end{equation}
Informally, $N_\sigma^-(j_0)$ (resp.\ $N_\sigma^+[\succ](j_0)$)
counts the elements of $N_\sigma^-$ (resp.\ $N_\sigma^+[\succ]$) in the $j_0$-th row of the matrix  associated with $\sigma$.
\begin{figure}[t]
\begin{tikzpicture}[scale=.5]
  \draw[step=1.0,gray,very thin,xshift=-0.5cm,yshift=-0.5cm] (1,1) grid (11,11);
\begin{scope}
  \draw[text=orange] (2,7)node{$1$}; \draw (4,6)node{$1$}; \draw
  (7,2)node{$1$}; \draw (9,1)node{$1$}; \draw[text=orange]
  (6,10)node{$1$}; \draw[text=orange] (8,9)node{$1$};
  \draw[text=orange] (10,8)node{$1$}; \draw[text=orange]
  (3,5)node{$1$}; \draw (5,4)node{$1$}; \draw[text=orange]
  (1,3)node{$1$}; \draw[thick,xshift=-0.5cm,yshift=-0.5cm]
  (1,8)--(4,8)--(4,6)--(6,6)--(6,4)--(8,4)--(8,1);
  \draw[dashed,thick,xshift=-0.5cm,yshift=-0.5cm] (4,11)--(4,8)
  (4,8)-- (11,8) (6,11)--(6,6) (6,6)--(11,6) (8,11)--(8,4)
  (8,4)--(11,4) (1,6)--(4,6) (4,6)--(4,1) (1,4)--(6,4) (6,4)--(6,1);
  \draw[thick,xshift=-0.5cm,yshift=-0.5cm,color=orange]
  (1,1)--(1,8)--(4,8)--(4,1)--(1,1)
  (11,11)--(11,8)--(4,8)--(4,11)--(11,11);
\end{scope}
\end{tikzpicture}
\caption{$U_{\sigma}(2)$ (entries equal to $1$ in the left orange
  rectangle) and $U_{\sigma}^{-1}(2)$ (entries equal to $1$ in the
  right orange rectangle) for $\eta = (3,2,2,3)$ and
  $\sigma = 681\02435179$}
\label{fig:U}
\end{figure}

\begin{lem} \label{Lem>} Let $\eta$ be a composition of $n$ and
  $\sigma \in S^{\eta}$. Then
\begin{align*}
    |\Nd^+_\sigma [\succ]| = \imv(\E(\pi_\eta(\sigma^{-1}))) + |\Nd^-_\sigma| + \iexc({\sigma})
\end{align*}
\end{lem}
\begin{proof}
 For a fixed excedance $l_0\in\Exc(\pi_{\eta}(\sigma^{-1}))$, write   
$j_0:= \sigma^{-1}(l_0)$ and set
$$ M_\sigma^=(j_0):=\{(j_0,\sigma(i)) : \sigma(i)<\sigma(j_0), i<j_0,
\pi_{\eta}(j_0)=\pi_{\eta}(i), \pi_{\eta}(\sigma(i))< \pi_{\eta}(i)\}$$ and
$$ M_\sigma^>(j_0):= \{(j_0,\sigma(i)) : \sigma(i)<\sigma(j_0), i<j_0,
\pi_{\eta}(j_0)<\pi_{\eta}(i), \pi_{\eta}(\sigma(i))< \pi_{\eta}(i)\}. $$ We prove the lemma in four steps.
\begin{enumerate}[1.]
\item\label{one} $\imv(\E(\pi_\eta(\sigma^{-1})))   = \sum_{(j_0,\sigma(j_0))\in [\succ]}  \left( |M_\sigma^=(j_0)| +  |M_\sigma^>(j_0)|\right)$.
  \smallskip
    \item\label{two} $|M_\sigma^=(j_0)| +  |M_\sigma^>(j_0)| + |N_\sigma^- (j_0)| +1 = |U_\sigma(\pi_\eta(j_0))|$. \smallskip
    \item\label{three} $|U_\sigma(\pi_\eta(j_0))| = |U^{-1}_\sigma(\pi_\eta(j_0))|$. \smallskip
    \item\label{four}   $|U^{-1}_\sigma(\pi_\eta(j_0))|=|N_\sigma^+[\succ](j_0)|$. \smallskip
\end{enumerate}

To prove  \ref{one} and \ref{two} we will use the following facts.
\begin{enumerate}[(i)]
\item\label{i} For $i<j$ we have $\pi_{\eta}(i)\leq \pi_{\eta}(j)$.
\item\label{ii} $\sigma$ is descending on every block-row, i.e.\ if
  $i\neq j$ with $\sigma(j)<\sigma(i)$ and
  $\pi_{\eta}(i)=\pi_{\eta}(j)$ then $j<i$.
\end{enumerate}

\smallskip
\noindent\emph{Proof of \ref{one}.} The idea here is to write the number of weak inversions of the exceeding word of $\pi_{\eta}(\sigma^{-1})$ as a sum of  equal pairs and strict inversions.  These are, in turn, refined according to the second element of the pair. Indeed, given $l_0$ and $j_0$ as before, using \eqref{ii} and setting $l:=\sigma(i)$ in the definition of $M_\sigma^=(j_0)$, we get
\begin{align}| M_\sigma^=(j_0)|&=|\{(j_0,\sigma(i)) : \sigma(i)<\sigma(j_0), i<j_0,
                                 \pi_{\eta}(j_0)=\pi_{\eta}(i), \pi_{\eta}(\sigma(i))< \pi_{\eta}(i)\}|\nonumber \\
  &=
   |\{(\sigma(i),j_0) : \sigma(i)<\sigma(j_0),
    \pi_{\eta}(j_0)=\pi_{\eta}(i), \pi_{\eta}(\sigma(i))< \pi_{\eta}(i)\}|\label{winv1}\\
  &=| \{(l,l_0) : l<l_0,
 \pi_{\eta}(\sigma^{-1}(l_0))=\pi_{\eta}(\sigma^{-1}(l)), \pi_{\eta}(l)< \pi_{\eta}(\sigma^{-1}(l))\}|.\nonumber
\end{align}
Similarly,
\begin{align}| M_\sigma^>(j_0)|&=  |\{(j_0,\sigma(i)) : \sigma(i)<\sigma(j_0), i<j_0,
 \pi_{\eta}(j_0)<\pi_{\eta}(i), \pi_{\eta}(\sigma(i))< \pi_{\eta}(i)\}|\nonumber \\
  &=  |\{(\sigma(i),j_0) : \sigma(i)<\sigma(j_0),\pi_{\eta}(j_0)<\pi_{\eta}(i), \pi_{\eta}(\sigma(i))< \pi_{\eta}(i)\}|
 \label{winv2}\\
  &= |\{(l,l_0) : l<l_0, \pi_{\eta}(\sigma^{-1}(l_0))<\pi_{\eta}(\sigma^{-1}(l)), \pi_{\eta}(l)< \pi_{\eta}(\sigma^{-1}(l))\}|.\nonumber
\end{align}
The claim follows, as 
$$\imv(\E(\pi_\eta(\sigma^{-1})))=\sum_{l_0}|\{(l,l_0) : l<l_0, \pi_{\eta}(\sigma^{-1}(l_0))\leq\pi_{\eta}(\sigma^{-1}(l)), \pi_{\eta}(l)< \pi_{\eta}(\sigma^{-1}(l))\}|,$$ where the sum ranges over $l_0\in\Exc(\pi_\eta(\sigma^{-1}))$.

\begin{figure}[t]
\begin{tikzpicture}[scale=.5]
  \draw[step=1.0,gray,very thin,xshift=-0.5cm,yshift=-0.5cm] (1,1) grid (11,11);
\begin{scope}
  \draw (10,2)node{$1$}; \draw (9,7)node{$1$}; \draw (8,10)node{$1$};
  \draw (7,6)node{$1$}; \draw (6,8)node{$1$}; \draw (5,4)node{$1$};
  \draw (4,9)node{$1$}; \draw (3,1)node{$1$}; \draw (2,3)node{$1$};
  \draw (1,5)node{$1$}; \draw[thick,xshift=-0.5cm,yshift=-0.5cm]
  (4,11)--(4,8)--(6,8)--(6,6)--(8,6)--(8,4)--(11,4);
  \draw[dashed,thick,xshift=-0.5cm,yshift=-0.5cm] (1,8)--(4,8) (4,8)--
  (4,6) (4,6)--(6,6) (6,6)--(6,4) (6,11)--(6,8) (8,11)--(8,4)
  (1,6)--(4,6) (4,6)--(4,1) (1,4)--(6,4) (6,4)--(6,1) (6,8)--(11,8)
  (8,6)--(11,6) (6,4)--(8,4) (8,4)--(8,1); \draw
  (7,8)node[text=orange]{$=$}; \draw (7,7)node[text=orange]{$>$};\draw (7,10)node[text=orange]{$>$};
\end{scope}
\end{tikzpicture}

\caption{For $\eta=(3,2,2,3)$ and $\sigma^{-1} = 846971521\03$, pick $l_0=5$, so  $j_0=\sigma^{-1}(l_0)=7$.
 The cells
  corresponding to the  elements of the set in \eqref{winv1} are
  marked with orange symbols ``='' and those corresponding to the
  elements of the set in \eqref{winv2} are marked by  orange symbols
  ``$>$''.}
  \label{fig:bijection_=,>}
\end{figure}

\smallskip
\noindent \emph{Proof of \ref{two}.} We partition $U_\sigma(\pi_\eta(j_0))=\{ (i,\sigma(i)) : \pi_\eta(j_0)\leq \pi_\eta(i), \pi_\eta(\sigma(i))<\pi_\eta(j_0) \}$ as follows:
\begin{align*}
    &\underbrace{\{ (i,\sigma(i)) : \sigma(i)<\sigma(j_0), i<j_0 \}\cap U_\sigma(\pi_\eta(j_0))}_{=: U^1_\sigma(\pi_\eta(j_0))} \\
    \cup\ &\underbrace{\{ (i,\sigma(i)) : \sigma(i)<\sigma(j_0), i>j_0 \}\cap U_\sigma(\pi_\eta(j_0))}_{=: U^2_\sigma(\pi_\eta(j_0))} \\
    \cup\ &\underbrace{\{ (i,\sigma(i)) : \sigma(i)>\sigma(j_0), i>j_0 \}\cap U_\sigma(\pi_\eta(j_0))}_{=: U^3_\sigma(\pi_\eta(j_0))}\\
    \cup\ &\{ (j_0,\sigma(j_0)) \},
\end{align*}
see Figure \ref{fig:U-subdivision} for an example.
Our goal is to rewrite the cardinalities of each of the
$U_{\sigma}^i(\pi_{\eta}(j_0))$.

\small
\begin{align*}
|U_{\sigma}^1(\pi_{\eta}(j_0))|&= |\{ (i, \sigma(i)):  \sigma(i)< \sigma(j_0), i<j_0, \pi_{\eta}(j_0)\leq \pi_{\eta}(i),\,\pi_{\eta}(\sigma(i))<\pi_{\eta}(j_0), \pi_{\eta}(\sigma(j_0))<\pi_{\eta}(j_0) \}| \\
&\overset{(i)}{=} |\{ (i, \sigma(i)):  \sigma(i)< \sigma(j_0), i<j_0, \pi_{\eta}(j_0)= \pi_{\eta}(i),\, \pi_{\eta}(\sigma(i))<\pi_{\eta}(j_0), \pi_{\eta}(\sigma(j_0))<\pi_{\eta}(j_0) \}|\\
&=  |\{ (j_0, \sigma(i)):  \sigma(i)< \sigma(j_0), i<j_0, \pi_{\eta}(j_0)= \pi_{\eta}(i),\,
 \pi_{\eta}(\sigma(i))<\pi_{\eta}(i), \pi_{\eta}(\sigma(j_0))<\pi_{\eta}(j_0) \}|\\
                               &= |M_{\sigma}^=(j_0)|,
                                 \end{align*}
\normalsize
Similarly,
\small
\begin{align*}
  |U_{\sigma}^2(\pi_{\eta}(j_0))|= &|\{ (i, \sigma(i)):  \sigma(i)<
                                     \sigma(j_0), j_0<i,
                                     \pi_{\eta}(j_0)\leq
                                     \pi_{\eta}(i),\,
                                     \pi_{\eta}(\sigma(i))<\pi_{\eta}(j_0), \pi_{\eta}(\sigma(j_0))<\pi_{\eta}(j_0) \}| \\
  \overset{(ii)}{=} & |\{ (j_0, \sigma(i)):  \sigma(i)< \sigma(j_0), j_0<i, \pi_{\eta}(j_0)< \pi_{\eta}(i),\,
                                    \pi_{\eta}(\sigma(i))<\pi_{\eta}(j_0), \pi_{\eta}(\sigma(j_0))<\pi_{\eta}(j_0) \}|\\
  = &|M_{\sigma}^>(j_0)|.
\end{align*}
\normalsize
Finally,
\small
\begin{align*}
  |U_{\sigma}^3(\pi_{\eta}(j_0))|=& |\{ (i, \sigma(i)):  \sigma(j_0)< \sigma(i), j_0<i, \pi_{\eta}(j_0)\leq \pi_{\eta}(i),\, \pi_{\eta}(\sigma(i))<\pi_{\eta}(j_0), \pi_{\eta}(\sigma(j_0))<\pi_{\eta}(j_0) \}| \\
  \overset{(i)}{=}& |\{ (j_0, \sigma(i)):  \sigma(j_0)< \sigma(i), j_0<i, \pi_{\eta}(\sigma(i))< \pi_{\eta}(j_0)\}|\\
  = &|\{ (j_0, h):  \sigma(j_0)< h, j_0<\sigma^{-1}(h), \pi_{\eta}(j_0)> \pi_{\eta}(h) \}|\\
  = &|N_{\sigma}^-(j_0)|,
\end{align*}\normalsize
where $h:=\sigma(i)$.

\begin{figure}[t]
      \begin{tikzpicture}[scale=.5]
  \draw[step=1.0,gray,very thin,xshift=-0.5cm,yshift=-0.5cm] (1,1) grid (11,11);
\begin{scope}
  \draw (6,10)node{$1$}; \draw (8,9)node{$1$}; \draw (10,8)node{$1$};
  \draw (2,7)node{$1$}; \draw (4,6)node{$1$}; \draw (3,5)node{$1$};
  \draw (5,4)node{$1$}; \draw (1,3)node{$1$}; \draw (7,2)node{$1$};
  \draw (9,1)node{$1$}; \draw[thick,xshift=-0.5cm,yshift=-0.5cm]
  (1,8)--(4,8)--(4,6)--(6,6)--(6,4)--(8,4)--(8,1);
  \draw[dashed,thick,xshift=-0.5cm,yshift=-0.5cm] (4,11)--(4,8)
  (4,8)--(11,8) (6,11)--(6,6) (6,6)--(11,6) (8,11)--(8,4)
  (8,4)--(11,4) (1,6)--(4,6) (4,6)--(4,1) (1,4)--(6,4) (6,4)--(6,1);
  \draw[thick,xshift=-0.5cm,yshift=-0.5cm,color=orange]
  (1,1)--(1,6)--(6,6)--(6,1)--(1,1); \draw[ultra
  thick,xshift=-0.5cm,yshift=-0.5cm,color=gray] (-1,5)--(5,5)--(5,12)
  (-1,4)--(5,4)--(5,0) (6,0)--(6,4)--(13,4); \draw
  (-1.7,5.5)node[text=darkgray]{$U_{\sigma}^1(\pi_{\eta}(j_0))$};
  \draw
  (-1.7,2.5)node[text=darkgray]{$U_{\sigma}^2(\pi_{\eta}(j_0))$};
  \draw
  (12.7,2.5)node[text=darkgray]{$U_{\sigma}^3(\pi_{\eta}(j_0))$};
\end{scope}
\end{tikzpicture}
  \caption{Let $\eta=(3,2,2,3)$, $\sigma= 68102435179$ and $j_0=7$. $U_{\sigma}^i(\pi_{\eta}(j_0))$, $i\in [3],$ is given by the entries equal to $1$ in the regions indicated by the grey lines intersected with the orange rectangle}
  \label{fig:U-subdivision}
\end{figure}
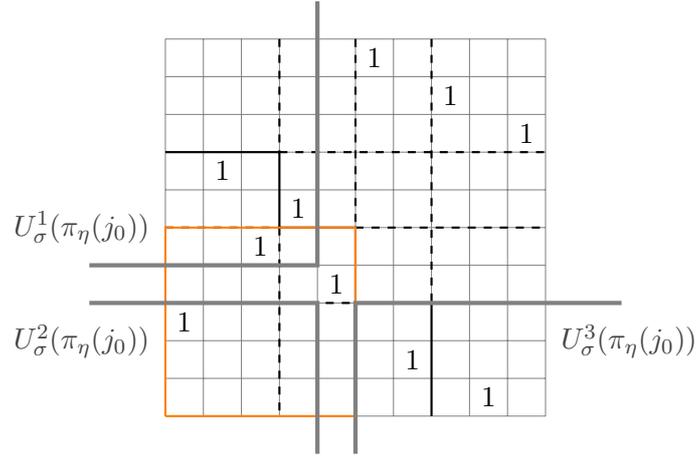

Therefore we obtain
\begin{align*}
|U_\sigma(\pi_\eta(j_0))| &= |U^1_\sigma(\pi_\eta(j_0))|+|U^2_\sigma(\pi_\eta(j_0))|+|U^3_\sigma(\pi_\eta(j_0))|+1\\
&=|M_\sigma^=(j_0)| +  |M_\sigma^>(j_0)| + |N_\sigma^- (j_0)| +1.
\end{align*}

\smallskip

 \noindent\emph{Proof of \ref{three}.} For $n_1,n_2\in \mathbb{N}$ and a permutation matrix divided into blocks  
$$
\left[
\begin{array}{c|c}
 B_{11} & B_{12} \\ \hline
 B_{21} & B_{22}
\end{array}\right]
$$
where each block $B_{ij}$ is an $n_i\times n_j$ matrix,  
we  let $l\in\mathbb N_0$ denote the number of entries equal to $1$ in
$B_{21}$. That is, $l=|\{ (i,\sigma(i)):
\sigma(i)\leq n_1  <i\}$ which is also the number of entries
equal to $1$ in $B_{12}$, while the number of entries equal to $1$ in $B_{11}$ is $n_1-l$.

\smallskip

\noindent \emph{Proof of \ref{four}.} For an excedance $l_0\in\Exc(\pi_\eta(\sigma^{-1}))$ and $j_0=\sigma^{-1}(l_0)$, we have
\begin{align*}
|U_{\sigma}^{-1}(\pi_{\eta}(j_0))| = |\{ (i, \sigma(i)) : \pi_{\eta}(i)<\pi_{\eta}(j_0), \pi_{\eta}(j_0)\leq \pi_{\eta}(\sigma(i)), \pi_{\eta}(\sigma(j_0))< \pi_{\eta}(j_0) \}|.
\end{align*}
Since $\pi_{\eta}(\sigma(j_0))\leq \pi_{\eta}(\sigma(i))$, it follows that $\sigma(j_0)<\sigma(i)$, Thus the above is equal to
\begin{align*}
 &|\{ (j_0, \sigma(i)) : \sigma(j_0)<\sigma(i), i<j_0,  \pi_{\eta}(j_0)\leq \pi_{\eta}(\sigma(i)), \pi_{\eta}(\sigma(j_0))< \pi_{\eta}(j_0) \}| \\
 =\ &|\{ (j_0, h) : \sigma(j_0)<h, \sigma^{-1}(h)<j_0,  \pi_{\eta}(j_0)\leq \pi_{\eta}(h), \pi_{\eta}(\sigma(j_0))< \pi_{\eta}(j_0) \}| \\
 =\ &|N_{\sigma}^+[\succ](j_0)|,
\end{align*}
proving \ref{four}.

\smallskip

For $j_0=\sigma^{-1}(l_0)$, where $l_0\in\Exc(\pi_\eta(\sigma^{-1})$, it now follows that
\begin{align*}
    |N_\sigma^+[\succ](j_0)| &\overset{4.}{=} |U^{-1}_\sigma(\pi_\eta(j_0))| \overset{3.}{=} |U_\sigma(\pi_\eta(j_0))| \overset{2.}{=} |M_\sigma^=(j_0)| +  |M_\sigma^>(j_0)| + |N_\sigma^- (j_0)| +1.
\end{align*}

Therefore, by Eq.~\eqref{eq:nsplus},
\begin{align*}
 | N_\sigma^+[\succ]|&= \sum_{(j,\sigma(j))\in[\succ]}  |N_\sigma^+[\succ](j)|=\sum_{(j,\sigma(j))\in[\succ]} |M_\sigma^=(j)| +  |M_\sigma^>(j)| + |N_\sigma^- (j)| +1\\
                                &=\imv(\E(\pi_\eta(\sigma^{-1}))) + |\Nd^-_\sigma| + \iexc({\sigma}),\end{align*}
                              where the latter equality follows from \ref{one}, Eq.~\eqref{eq:nsminus} and the definition of $\iexc$.
\end{proof}

\begin{proof}[Proof of Theorem~\ref{Euler-Mahonian}] Combining Lemma~\ref{Lem=<} and Lemma~\ref{Lem>}, for $\sigma\in S^{\eta}$ we get
  $$\denh(\pi_\eta(\sigma^{-1}))=\den(\sigma). $$
 The theorem follows, as by definition $\iexc(\sigma)=\exc(\pi_\eta(\sigma^{-1}))$ and the map defined in Eq.~\eqref{eq:bijection} is a bijection between $S^\eta$ and $S_\eta$.
  \end{proof}

Let $\eta=(\eta_1,\dots, \eta_r)$ be a composition of $n$. Theorems
\ref{Theorem_Han} and \ref{Euler-Mahonian} imply that  
 the genus zeta function of
 the local hereditary order $\Theta=\Theta^\eta$ can be rewritten in terms of the pair of statistics $(\maj,\des)$.
 \begin{cor}\label{local}
 \begin{equation*}
   Z_{\Theta^\eta}(s)=\frac{\sum_{w\in S_\eta} q^{\maj(w)-ns\des(w)}}{\prod_{i=0}^{n-1}(1-q^{i-ns})}.
 \end{equation*}
 \end{cor}

 The next corollary follows directly from \cite[Proposition 2.12]{CarnevaleVoll/18} (see also \cite[Theorem 1.3]{CarnevaleVoll/18}) and establishes a reciprocity property for the genus
 zeta function of local hereditary orders whose associated composition is a
 \emph{rectangle}  (i.e.\ all its parts are equal).
 \begin{cor} Let  $r,m\in \mathbb N$ and $\eta=(\underbrace{m,
  \dots ,m}_{r})=:(m^r)$. Then
   \begin{equation*}
     Z_{\Theta^\eta}(s)\vert_{q\to q^{-1}}=(-1)^{rm}q^{\frac{rm(m-1)}{2}-mns} Z_{\Theta^\eta
     }(s).
   \end{equation*}
   If $\eta$ is not a rectangle, then $Z_{\Theta^\eta}(s)$ does
   not satisfy a functional equation of the form
    \begin{equation*}
     Z_{\Theta^\eta}(s)\vert_{q\to q^{-1}}=\pm q^{a-bs} Z_{\Theta^\eta}(s).
   \end{equation*}
   for $a,b\in \mathbb N_0$.
 \end{cor}
 
It would be interesting to establish a purely algebraic explanation of this result.

 \section{Signed and even-signed permutations}\label{sec:bd}

In this section, we define signed analogues of the Denert statistic and show that they
are, together with the number of absolute excedances, equidistributed
with the the flag major index and the number of flag descents over the
hyperoctahedral groups. For a suitable definition of type $D$ descents
and major indices, we define a type $D$ Denert statistic and number of
excedances which are equidistributed over the even-signed
permutations.

\subsection{Euler-Mahonian statistics on $B_n$}

Let $B_n$ denote the group of signed permutations on $n$ letters, i.e.\ permutations of
the set $[-n,n]$ such that $\sigma(-i)=-\sigma(i)$ for $i\in [0,n]$. 
For a signed permutation $\sigma\in B_n$, we use the window notation
$\sigma= \sigma(1)\dots \sigma(n)$. By slight abuse of notation, we
 denote by $\des(\sigma)$  and $\maj(\sigma)$ the type $A$ descent and major index statistics of the signed permutation $\sigma$, as defined in Section~\ref{sec:preliminaries}.

Well-known statistics on signed permutations (see for example
\cite{AdBrRo/01}) include the negative statistics
\begin{align*}
&\negg(\sigma)=|\{i\in [n] : \sigma(i)<0\}|, \\
&\ndes(\sigma)= \des(\sigma) + \negg(\sigma) \quad \text{and} \quad \nmaj(\sigma)= \maj(\sigma) - \sum_{\sigma(i)<0} \sigma(i)
\end{align*}
and the flag statistics
\begin{align*}
\fdes(\sigma)=2 \des(\sigma)+\chi(\sigma(1)<0) \quad \text{and} \quad \fmaj(\sigma)=2 \maj(\sigma)+\negg(\sigma),
\end{align*}
where 
\begin{align*}
\chi(\sigma(1)<0)= \begin{cases} 
1 \quad \text{ if } \sigma(1)<0 \\
0 \quad \text{ otherwise}.
\end{cases} 
\end{align*}
In \cite{AdBrRo/01} the two pairs of statistics $(\nmaj, \ndes)$ and
$(\fmaj, \fdes)$ were  shown to be equidistributed.
\begin{theorem}\cite[Corollary 4.5]{AdBrRo/01} \label{equidistr_neg_flag}
\begin{align*}
  \sum_{\sigma\in B_n} q^{\nmaj(\sigma)} t^{\ndes(\sigma)} = \sum_{\sigma\in B_n} q^{\fmaj(\sigma)} t^{\fdes(\sigma)}.
\end{align*}
\end{theorem}
Denert's statistic has been extended to signed permutations before (see, e.g.,~\cite{Fire/04}).
To the best of our knowledge, none of the type $B$ extensions previously considered gives rise, together with a suitable definition of
excedances, to an Euler-Mahonian pair in the sense of Theorem
\ref{equidistr_neg_flag}.

\begin{dfn}\cite[Definition 4.1]{Mongelli/15}
  For $\sigma\in B_n$, we define
  $|\sigma|= |\sigma(1)| \dots |\sigma(n)|\in S_n$. The absolute
  excedance number is
\begin{align*}
\excabs(\sigma) = \exc(|\sigma|) + \negg(\sigma).
\end{align*}
\end{dfn}
We define a  Denert statistic for signed permutations as follows.
\begin{dfn}
Let $\sigma\in B_n$. The negative Denert statistic is
\begin{align*}
\nden(\sigma) = \denh(|\sigma|) - \sum_{\sigma(i)<0} \sigma(i).
\end{align*}
\end{dfn}

The following theorem shows that the pairs of statistics $(\nden, \excabs)$ and $(\fmaj,\fdes)$ are equidistributed over the hyperoctahedral groups.
\begin{theorem}
\begin{align*}
\sum_{\sigma\in B_n} q^{\nden(\sigma)} t^{\excabs(\sigma)} = \sum_{\sigma\in B_n} q^{\fmaj(\sigma)} t^{\fdes(\sigma)}.
\end{align*}
\end{theorem}
\begin{proof}
  Writing a signed permutation as a product of an element in the
  symmetric group and a sign vector yields:
\begin{align*}
\sum_{\sigma\in B_n} q^{\nden(\sigma)} t^{\excabs(\sigma)} &= \sum_{\sigma\in B_n} q^{\denh(|\sigma|)} q^{- \sum_{\sigma(i)<0} \sigma(i)} t^{\exc(|\sigma|)} t^{\negg(\sigma)} \\
&= \left( \sum_{\sigma\in S_n} q^{\denh(\sigma)} t^{\exc(\sigma)} \right) \left( \sum_{J\subseteq [n]} \sum_{j\in J} q^j t \right)\\
&= \left( \sum_{\sigma\in S_n} q^{\maj(\sigma)} t^{\des(\sigma)} \right) \left( \sum_{J\subseteq [n]} \sum_{j\in J} q^j t \right) \\
&= \sum_{\sigma\in B_n} q^{\nmaj(\sigma)} t^{\ndes(\sigma)},
\end{align*}
where the penultimate equality follows from Theorem~\ref{Theorem_Han}.
The claim now follows by Theorem~\ref{equidistr_neg_flag}.
\end{proof}

\subsection{Euler-Mahonian statistics on $D_n$}

We define a type $D$ analogue of Denert's statistic which, together
with a suitable definition of an excedance statistic, forms an Euler-Mahonian pair. 
The Coxeter group $D_n$ is the
subgroup of $B_n$ of even-signed permutations,
\begin{align*}
D_n = \{ \sigma \in B_n : \negg(\sigma)\equiv 0 \mod 2\}.
\end{align*}
A negative descent set on $D_n$ and corresponding descent number and
major index were defined in \cite{Biagioli/03}.
\begin{dfn} \cite[Section 3.1]{Biagioli/03} Let $\sigma\in D_n$. The type $D$ negative descent set of $\sigma$ is
$$
\DNeg(\sigma) = \{i\in [n] : \sigma(i)<-1\} \quad \text{and} \quad \dneg(\sigma)= |\DNeg(\sigma)|.
$$
The corresponding descent and major index statistics are
\begin{align*}
\ddes(\sigma) &= \des(\sigma) + \dneg(\sigma), \\
\dmaj(\sigma) &= \maj(\sigma) - \sum_{i\in\DNeg(\sigma)} \sigma(i) - \dneg(\sigma).
\end{align*}
\end{dfn}

\begin{dfn} 
For $\sigma\in D_n$, we define the number of type $D$ excedances to be
\begin{align*}
\dexc(\sigma) := \exc(|\sigma|) + \dneg(\sigma).
\end{align*}
\end{dfn}
Note that the number of type $D$ excedances of $\sigma\in D_n$
differs from the number of absolute excedances of $\sigma$
if $\sigma(i)=-1$ for some $i\in [n]$.

\begin{dfn}
We define Denert's statistic for even-signed permutations as
\begin{align*}
\dden(\sigma) &:= \denh(|\sigma|) - \sum_{i\in\DNeg(\sigma)} \sigma(i) - \dneg(\sigma) = \denh(|\sigma|) + \nsp(\sigma),
\end{align*}
where $\nsp(\sigma) := |\{(i,j)\in [n]\times [n] : i<j, \sigma(i) + \sigma(j) <0\}|$ denotes the negative sum pairs.
\end{dfn}

The next theorem shows that $(\dden, \dexc)$ and $(\dmaj, \ddes)$ are equidistributed over the even-signed permutations.

\begin{theorem}
\begin{align*}
\sum_{\sigma\in D_n} q^{\dden(\sigma)} t^{\dexc(\sigma)} = \sum_{\sigma\in D_n} q^{\dmaj(\sigma)} t^{\ddes(\sigma)}.
\end{align*}
\end{theorem}
\begin{proof}
Write  $D_n$ as $$D_n = \bigcup_{\pi\in S_n} \{ \tau\pi : \tau\in T \},$$ where $T = \{ \tau\in D_n : \des(\tau)=0 \}$ and the union is disjoint. Then
\begin{align}\label{eq:dd}
\sum_{\sigma\in D_n} q^{\dden(\sigma)} t^{\dexc(\sigma)}
&= \sum_{\pi\in S_n}\sum_{\tau \in T}q^{\denh(|\tau\pi|)-\sum_{i\in\DNeg(\tau\pi)} \tau\pi(i) - \dneg(\tau\pi)} t^{\exc(|\tau\pi|)+\dneg(\tau\pi)}.
\end{align}
It is easy to see that  $ \sum_{i\in\DNeg(\tau\pi)} \tau\pi(i)  =  \sum_{i\in\DNeg(\tau)} \tau(i)  $ and $\dneg(\tau\pi)=\dneg(\tau)$   for any $\pi\in S_n$ and $\tau \in T$. Thus Eq.~\eqref{eq:dd} is
equal to
\begin{align*}
& \sum_{\tau \in T} q^{-\sum_{i\in\DNeg(\tau)} \tau(i) - \dneg(\tau)} t^{\dneg(\tau)} \sum_{\pi\in S_n} q^{\denh(\pi)} t^{\exc(\pi)}.
\end{align*}
By Theorem~\ref{Theorem_Han}, this is equal to
\begin{align*}
= & \sum_{\tau \in T} q^{-\sum_{i\in\DNeg(\tau)} \tau(i) - \dneg(\tau)} t^{\dneg(\tau)} \sum_{\pi\in S_n} q^{\maj(\pi)} t^{\des(\pi)} \\
= & \sum_{\sigma\in D_n} q^{\dmaj(\sigma)} t^{\ddes(\sigma)},
\end{align*}
which proves the theorem.
\end{proof}

\section{Final remarks}\label{sec:final}
\subsection{Hadamard products}
 By a formula due to MacMahon \cite[§462, Vol. 2, Ch. IV,
 Sect. IX]{MacMahon/04} and Theorem~\ref{Euler-Mahonian}, it turns out
 that genus zeta functions as in Corollary~\ref{local}, viewed as rational
 functions in $q$ and $q^{-ns}$ are closely related to Hadamard
 products of the rational functions expressing genus zeta functions of
 maximal orders (i.e.\ orders whose local type is a composition with one part). In the following, given rational functions $F(y)$ and
 $G(y)$, we denote with $(F\star G)(y)$ their Hadamard product.
 Then
 \begin{equation}\label{eq:hadamard}
 W_\eta(x,y)=(1-x^ny)\Ast_{i=1}^{r}W_{(\eta_i+1)}(x,y)=(1-x^ny)\Ast_{i=1}^{r}\left(\prod\limits_{0\leq j \leq \eta_i}\frac{1}{1-x^jy}\right),\end{equation}
 where the Hadamard product is taken with respect to $y$.

 At present, we are not aware of an
 algebraic interpretation, say in terms of factorisation of ideals in~$\Theta^\eta$, of the Hadamard product in
 Eq.~\eqref{eq:hadamard}.

 Certain orbit Dirichlet series exhibit a similar behaviour;
 cf.~\cite[Proposition 1.2]{CarnevaleVoll/18}. An algebraic framework for interpreting Hadamard products of closely related generating functions was recently developed by Gessel and Zhuang \cite{Gessel/18}.

 \subsection{Factorisation}

 It is well known that classical Eulerian polynomials over $S_n$ have
 all real, simple negative roots and that $-1$ is a root if and only
 $n$ is even; see \cite{frobenius/1910}. It was proved in \cite[Lemma 2.7]{CarnevaleVoll/18} that
 this generalises to a factorisation of the $q$-Carlitz polynomial for
 $n$ even. In the same paper, it also was conjectured that a similar
 factorisation result should hold for the polynomials giving the joint
 distribution of $(\des,\maj)$ over multiset permutations associated
 with compositions which are rectangles and satisfy certain conditions
 (see \cite[Conjecture
 B]{CarnevaleVoll/18}). Han's result Theorem~\ref{Theorem_Han} and Theorem~\ref{Euler-Mahonian}  allow
 for reformulations of this conjecture in terms of the pair of
 statistics $(\exc,\denh)$ over multiset permutations and in
 terms of Denert's original statistic over $\eta$-admissible permutations.
 More precisely, the conjecture revolves
 around the existence of so-called unitary factors of
 Euler-Mahonian polynomials. A nonconstant
 polynomial $f\in\mathbb Z[x,y]$ is called \emph{unitary} if there
 exists $F\in\mathbb Z[Y]$ such that $f(x,y)=F(x^ay^b)$ for some
 $a,b\in \mathbb N_0$ and all complex roots of $F$ have absolute value $1$.
 \begin{con}
   Let $\eta$ be a composition. Then the polynomial of the joint
   distribution of $(\denh,\exc)$ over $S_{\eta}$ has a unitary
   factor if and only if $\eta=(m^r)$ is a rectangle, with $r$ even
   and $m$ odd. In this case,
   $$\sum_{w\in S_{\eta}}x^{\denh(w)}y^{\exc(w)}=\left(1+x^{\frac{rm}{2}}y\right)f^{\eta}_0(x,y)$$
   where $f^{\eta}_0(x,y)$ has no unitary factor.

   \end{con}

 \begin{acknowledgements} We would like to thank Tobias Rossmann and Christopher Voll for mathematical conversations and helpful comments. This paper is part of the second author's PhD project, supervised by Christopher Voll. The second author is funded by the Deutsche Forschungsgemeinschaft DFG through grant no.~380258175.
\end{acknowledgements}

\bibliographystyle{plain}


\end{document}